\documentclass[a4paper,11pt]{article}
\usepackage{amsmath,amsthm,amssymb,amsfonts,color,stmaryrd}
\usepackage{hyperref,graphicx}

\newcommand{\C}{{\mathbb{C}}}          % \C       = complexos
\newcommand{\Proj}{{\mathbb{P}}}        % \Pro     = projectivo
\newcommand{\R}{{\mathbb{R}}}          % \R       = reais
\newcommand{\Z}{{\mathbb{Z}}}          % \Z       = inteiros

\newcommand{\g}{{\mathfrak{g}}}       %

\newcommand{\gll}{{\mathfrak{gl}}}
\newcommand{\sol}{{\mathfrak{o}}}

\newcommand{\gdois}{{\mathrm{G}_2}}

\newcommand{\SO}{{\mathrm{SO}}}
\newcommand{\SU}{{\mathrm{SU}}}

\newcommand{\rr}{\rightarrow}
\newcommand{\lrr}{\longrightarrow}

\newcommand{\calH}{{\cal H}}             %
\newcommand{\calP}{{\cal P}}             %
\newcommand{\calR}{{\cal R}}             %

\newcommand{\na}{{\nabla}}

\newcommand{\tr}[1]{{\mathrm{tr}}\,{#1}}

\newcommand{\Ad}[1]{{\mathrm{Ad}}\,{#1}}

\newcommand{\dx}{{\mathrm{d}}}

\newcommand{\inv}[1]{{#1}^{-1}}
\newcommand{\papa}[2]{\frac{\partial#1}{\partial#2}}

\newcommand{\vol}{{\mathrm{vol}}}
\newcommand{\Vol}{{\mathrm{Vol}}}
\newcommand{\ric}{{\mathrm{Ric}\,}}
\newcommand{\Scal}{{\mathrm{Scal}}}

\newtheorem{teo}{Theorem}[section]
\newtheorem{lemma}{Lemma}[section]
\newtheorem{coro}{Corollary}[section]
\newtheorem{prop}{Proposition}[section]

\newenvironment{rema}[1][Remark.]{\begin{trivlist}
\item[\hskip \labelsep {\bfseries #1}]}{\end{trivlist}}

\def\cyclic{\mathop{\kern0.9ex{{+}
\kern-2.2ex\raise-.28ex\hbox{\Large\hbox{$\circlearrowright$}}}}\limits}

\begin{document}

\title{Self-duality and associated parallel or cocalibrated $\gdois$ structures}

\author{R. Albuquerque}

\maketitle

\

\begin{abstract}

We find a remarkable family of $\gdois$ structures defined on certain principal $\SO(3)$-bundles $P_\pm\lrr M$ associated with any given oriented Riemannian 4-manifold $M$. Such structures are always cocalibrated. The study starts with a recast of the Singer-Thorpe equations of 4-dimensional geometry. These are applied to the Bryant-Salamon cons\-truction of complete $\gdois$-holonomy metrics on the vector bundle of self- or anti-self-dual 2-forms on $M$. We then discover new examples of that special holonomy on disk bundles over $\calH^4$ and $\calH^2_\C$, respectively, the real and complex hyperbolic space. Only in the end we present the new $\gdois$ structures on principal bundles.

\end{abstract}

\vspace*{12mm}

\noindent
{\bf Key Words:} self-dual metric, calibration, holonomy, $\gdois$ structure

\vspace*{1mm}
\noindent
{\bf MSC 2010:} Primary:  53C25, 53C38; Secondary: 53C20, 53C28, 53C29

\vspace*{22mm}

\markright{\sl\hfill  R. Albuquerque \hfill}

\section*{Introduction}

The group $\gdois$ of automorphisms of the octonions is equally characterised as the group of invariants of a certain 3-form $\phi\in\Lambda^3(\R^7)^*$. This Lie subgroup of $\SO(7)$ gives rise to a 7-dimensional special Riemannian geometry, whose basics are very well-known today. A $\gdois$ structure on a 7-manifold is given by a reduction of the manifold structure group to $\gdois$. It is equivalently given by a certain 3-form over the manifold. Ever since a thorough study by Bryant and Salamon came to light, in \cite{Bryant2,BrySal,Sal3}, the geometry of $\gdois$ structures has deserved much attention and led to various deep insights and questions.

Let $M$ be a 4-dimensional oriented Riemannian manifold. The present article finds a new family of $\gdois$ structures associated to $M$. They are defined on the total spaces of two natural principal $\SO(3)$-bundles $P_+$ and $P_-\lrr M$, abbreviated $P_\pm$, of \textit{oriented} orthonormal coframe basis $\{e^1,e^2,e^3\}$ of self-dual and, respectively, anti-self-dual 2-forms on $M$. 

The following gives immediately a particular case, say a preferred $\gdois$ structure within the new family. Writing the connection 1-form $\omega\in\Omega^1_{P_\pm}(\sol(3))$, induced from the Levi-Civita connection of $M$ on the vector bundle $\Lambda_\pm^2T^*M$, as
\begin{equation*}
 \omega=\left[\begin{array}{ccc}
              0 & -\omega^3 &\omega^2 \\
              \omega^3& 0 &-\omega^1 \\
              -\omega^2 &\omega^1 &0
              \end{array} \right]  
\end{equation*}
then a $\gdois$ structure 3-form $\phi$ on $P_\pm$ is defined by:
\[ \phi= \omega^1\wedge\omega^2\wedge\omega^3
        \mp(e^1\wedge\omega^1+e^2\wedge\omega^2+e^3\wedge\omega^3) \ .  \]
One may say that a basic knowledge of the theory up to Bianchi identity in 4-dimensional geometry is most sufficient in order to prove $\phi$ is coclosed.

The family of $\gdois$ cocalibrated metrics explicitly found is a natural variation of the above preferred structure. The abundance of these examples is consistent with an existence result on spin manifolds and the h-principle of cocalibrated structures, proved in \cite[Theorem 1.8]{CrowleyNordstrom}. It is also important for the construction of $\mathrm{Spin}(7)$ metrics on $P_\pm\times(-\epsilon,\epsilon)$, $\epsilon>0$, if one proceeds with the `Hitchin flow' technique. From another perspective, the cocalibration $(P_\pm,\phi)$ is quite surprising since it reveals a new kind of twistorial framework for the study of oriented Riemannian 4-manifolds and, therefore, also a potential for new functor relations between 4- and 7-dimensional geometry. More plainly, our result compares with the well-known theorem which says that every cotangent bundle \textit{is} a symplectic manifold.

We start our study with a recast of the theory of connections on principal coframe bundles and the Singer-Thorpe decomposition of the curvature tensor of a Riemannian 4-manifolds. We have given below a quite independent proof of this decomposition. These well-known results are used along the later proofs of the main theorems.

We also present an introduction to fundamental notions and equations of $\gdois$ geometry. Then we revisit the $\gdois$-holonomy metrics on $X_\pm=\Lambda^2_\pm T^*M$, constructed by R. Bryant and S. Salamon in \cite{BrySal,Sal3}, somehow willing to honour their discovery of true $\gdois$-holonomy. We compute the fundamental torsion equations of \cite{Bryant1,FerGray} on $X_\pm$, for $M$ anti-self-dual, or self-dual for the \textit{minus} case, which are finally related by an elementary lemma about two 1-variable dependent positive functions (throughout the paper we work in the smooth category). The torsion forms entail many new unsolved questions. As our computations are also accomplished for the bundle of self-dual 2-forms, we use results of C.~LeBrun (\cite{LeB1,LeB3}) to deduce that to every $K3$ surface with Calabi-Yau metric there corresponds a 2-parameter family of parallel $\gdois$ structures on $\Lambda^2_+T^*K3$.

Our last chapter contains the general equations of the new $\gdois$ structures on the manifolds $P_\pm$. These structures are always cocalibrated. There remain non-vanishing torsion forms, which we also find.

In particular, the family of cocalibrated structures on $P_-$ over $S^4$ or $\C\Proj^2$ may be chosen to be nearly parallel, with arbitrarily chosen positive $\|\dx\phi\|_\phi>0$.

The author acknowledges Anna Fino for her commenting of a first draft of this article. He greatly acknowledges the anonymous referees who very much improved a final version of the text.

\section{Riemannian 4-manifolds and {$\gdois$} structures on 7-manifolds}
\label{sec:one}

\subsection{Frame bundle and connection forms}
\label{sec1:subone}

We start by recalling some classical elements of differential and Riemannian geometry, which may be seen in many references such as \cite{Hel,KobNomi}. Introducing notation, given a manifold $Y$ and a vector bundle $E\rr Y$, we let $\Omega^p_Y(E)$ or $\Omega^p(Y,E)$ represent the space of smooth sections $\Gamma(Y;\Lambda^pT^*Y\otimes E)$. Also, we let $\Omega^p_Y=\Omega^p_Y(\R)$.

Let $M$ denote a smooth $n$-dimensional manifold and let $F^*M$ be the principal $\mathrm{GL}(n,\R)$-bundle of coframes. A coframe  $e\in F^*M$ is a linear isomorphism $(e^1,\ldots,e^n):T_mM\lrr\R^n$, $m\in M$. The natural Lie group right-action $(e,g)\mapsto R_g(e)=e\cdot g$ is defined by $e\cdot g=(\sum_jg^1_je^j,\ldots,\sum_jg^n_je^j)$, for $g\in \mathrm{GL}(n,\R)$.

Using the bundle projection $\pi:F^*M\lrr M$ we have a canonical $\R^n$-valued 1-form $\theta$ on $F^*M$, the so-called \textit{soldering} form. It gives a first example of a tautological form, as it is defined by
\begin{equation}
 \theta_e=e\circ \pi_*\ .
\end{equation}

Now suppose the manifold is endowed with a linear connection, that is, a covariant derivative $\na$ on the tangent bundle of $M$.

Given any local section $s=(e^1,\ldots,e^n):U\rr F^*M$ on an open subset $U\subset M$, we then have a matrix-valued 1-form $\omega$ induced by the covariant derivative: $\na e^i=\sum_je^j\otimes\omega^i_j$. In obvious notation we may write this as $s\,\omega^i_\cdot$, with such matrix 1-form existing on $U$.

Now a natural extension $\dx^\na$ of $\na$ as a differential operator on the relevant space leads us to the notion of the \textit{curvature} tensor $R^\na=(\dx^\na)^2$ and, locally, to a \textit{curvature} form $\rho^i_k$. Respectively, a $T^*M$-valued 2-form on $M$
\begin{equation}\label{curvature0}
 R^\na_{Z_1,Z_2}e^i=\na_{Z_1}\na_{Z_2}e^i-\na_{Z_2}\na_{Z_1}e^i-\na_{[Z_1,Z_2]}e^i,\quad\forall Z_1,Z_2\in TM\ ,
\end{equation}
and a Lie algebra $\gll(n,\R)$-valued 2-form on $U$
\begin{equation}\label{curvature1}
 \rho^i_k=\dx\omega^i_k+\sum_j\omega^j_k\wedge\omega^i_j\ .
\end{equation}
Of course, (\ref{curvature0}) and (\ref{curvature1}) are related by $R^\na e^i=s\rho^i_\cdot$ and, differentiating again, gives a Bianchi identity.

More important here is the fact that the connection can be completely described over the manifold $F^*M$. Indeed, there exists a unique globally defined $\omega\in \Omega^1(F^*M,\gll(n,\R))$ such that
\begin{equation}
 \na s=s\,s^*\omega,\quad \forall s\in \Omega^0(U,F^*M)\ , 
\end{equation}
and such that, for any \textit{fundamental} vertical vector field $V_e\in TF^*M,\ e\in F^*M$,
\begin{equation}
   \omega(V_e)=V\qquad 
   \bigl(\mbox{by definition},\ V_e=\left.\frac{\dx}{\dx t}\right\vert_0  e\cdot\mathrm{exp}(tV),\,\ V\in \gll(n,\R)\bigr)\ .
 \end{equation}
From this and the existence of time-dependent parallel sections we have that $H=\ker\omega$ is complementary to the vertical tangent subbundle $\ker\pi_*\subset TF^*M$. It follows easily that $R_g^*\omega=\Ad(\inv{g})\omega,\ \forall g\in \mathrm{GL}(n,\R)$. And, hence, that $\dx{R_g}(H_e)=H_{e\cdot g}$.

Let us recall the connection $\na$ is given on the tangent bundle of $M$. Here we must consider the \textit{torsion}, defined by $T^\na=\dx^\na1$. Letting $\breve{s}=(e_1,\ldots,e_n)$ denote a frame dual to the previous $s$, we may then define equivariantly an $\R^n$-valued 2-form $\tau$ on $F^*M$, vanishing on vertical directions and such that $T^\na=\breve{s}\,s^*\tau^t$.

The connection 1-form of $TM$ is $-\omega^t$, i.e. it satisfies $\na e_i=-\sum_je_j\,\omega^j_i$ or just $\na\breve{s}=-\breve{s}\,s^*\omega^t$, because simply one requires $\na1=0$. The following are two fundamental equations due to \'E.~Cartan regarding the torsion and the \textit{curvature} of any linear connection on the principal bundle of coframes.
\begin{prop}[Cartan structural equations]
We have
\begin{equation}\label{cartan1}
\tau=\dx\theta+\theta\wedge\omega\ ,\qquad\qquad   \rho=\dx\omega+\omega\wedge\omega\ .
\end{equation}
\end{prop}
\begin{proof}
In order to readily establish the theory, we give the proof with as much detail as possible. First the map $\breve{s}\, s^*\theta^t=\sum_j e_j\theta^js_*=\sum_je_je^j=1_{|U}$ is the identity endomorphism of $TM$. Then
we have $\breve{s}\,s^*\tau^t=
\dx^\na1_{|U}=\dx^\na(\breve{s}\, s^*\theta^t)= \breve{s}(-s^*\omega^t\wedge s^*\theta^t+\dx s^*\theta^t)= \breve{s}\,s^*(\theta\wedge\omega+\dx\theta)^t$. Let us see that for a vertical direction $V_e$, we have $(\dx\theta+\theta\wedge\omega)(V_e,\cdot)=0$. This is trivial if the second entry is vertical too, so we consider a lift $Z_{e\cdot g}=\dx R_g(s_*(Z))$ of $Z\in TM$, with $s$ a section passing by $e$, and compute:
\begin{equation*}
 [V,Z]_e \ =\ \lim_{t\rr0}\frac{1}{t}(\dx R_{\mathrm{exp}(-tV)}(Z_{e\cdot\mathrm{exp}(tV)})-Z_e)\ =\ 0 , 
\end{equation*}
\begin{eqnarray*}
 {\lefteqn{ (\dx\theta+\theta\wedge\omega)(V_e,Z_e) \ =\   \dx(\theta(Z))(V_e)-\dx(\theta(V))(Z_e)-\theta([V,Z])-\omega(V_e)\theta(Z_e)  }}\\
  &=& \left.\frac{\dx}{\dx t}\right\vert_0(\theta Z)_{e\cdot\mathrm{exp}(tV)}-V(e(Z))\ =\ \left.\frac{\dx}{\dx t}\right\vert_0e\cdot\mathrm{exp}(tV)(Z)-V(e(Z)) \ =\ 0 .
\end{eqnarray*}
Regarding the curvature equation in \eqref{cartan1}, with the above coframe we find $R^\na s=\dx^\na(s\, s^*\omega)=s\,s^*(\omega\wedge\omega+\dx\omega)$ which by definition is $R^\na s=s\, s^*\rho$, as in \eqref{curvature1}. Arguments such as the previous yield $\rho(V,\cdot)=0$, proving $\rho$ is well-defined and equivariant.
\end{proof}
We recall that $\theta,\omega$, and hence $\tau$ and $\rho$, are global differential forms on $F^*M$.

A connection is said to be \textit{reducible} to a principal $G$-subbundle $Q$ of $F^*M$, where $G$ is a Lie subgroup of the general linear group, if $\ker\omega_{|Q}\subset TQ$.

The previous classical theory extends to any vector bundle $X\lrr M$ which is associated to a coframe principal $G$-bundle $Q\lrr M$. This is, when it is given a representation $\sigma:G\rr\mathrm{GL}(V)$, where $V$ is a vector space, so that we may write $X=Q\times_\sigma V$. This means a vector in $X$ identifies with a pair $(q,f)\in Q\times V$ or any representative of its equivalence class, $(qg,\sigma(\inv{g})f)$, for $g\in G$, the usual orbit of $G$. If $s$ is a section of $Q$ on an open set $U\subset M$ and $f$ is any $V$-valued function on $U$, then $f$ determines a unique $G$-equivariant function $\hat{f}:\inv{\pi_{|Q}}(U)\rr V$ such that $f=\hat{f}\circ s$; with equivariant meaning that $\sigma(\inv{g})\hat{f}(s)=\hat{f}(sg)$, $\forall g\in G$. Reciprocally, any equivariant function on $Q$ determines a section of $X\lrr M$. Finally, we covariant differentiate sections of $X$ through the class-independent formula
\begin{equation}
 \na_Z(s,f)=(s,\hat{\sigma}\cdot s^*\omega(Z)f+\dx f(Z)),\quad\forall Z\in TM\ , 
\end{equation}
where $\hat{\sigma}:\g\lrr\gll(V)$ is the induced map from $\sigma$. To see this is well-defined on $X$ it is necessary to prove first $(sg)^*\omega=\Ad(\inv{g})s^*\omega+\inv{g}\dx g$, where $g$ is any $G$-valued function defined on the domain of $s$. However, we shall not really need this formula in what follows.

Now we suppose $M$ is also an oriented Riemannian manifold with metric $\mathrm{g}=\langle\ ,\ \rangle$. Then there is a canonical torsion-free metric connection, the Levi-Civita connection, and all the above remains true on the principal $\SO(n)$-bundle $F_\circ^*M$ of oriented orthonormal coframes. Because $\omega$ defines a metric connection, the matrix of 1-forms $\omega^i_j$ is skew-symmetric. Moreover, any 1-form $\kappa=\tr(\kappa\circ1)$ or 0-section of $T^*M$ satisfies $0=\dx^2\kappa=\tr(R^\na\kappa\wedge1)$ for a torsion-free connection. This leads to the so-called \textit{first} Bianchi identity
\begin{equation} \label{firstBianchiidentity}
R^\na_{e_\alpha,e_\beta}e^\gamma+R^\na_{e_\beta,e_\gamma}e^\alpha+R^\na_{e_\gamma,e_\alpha}e^\beta=0\ .
\end{equation}

\subsection{Self-duality on Riemannian 4-manifolds}
\label{sec1:subtwo}

Recall that a star operator $*$ is defined on $\Lambda^2(\R^4)^*$ by $\alpha\wedge*\beta=\langle\alpha,\beta\rangle\vol$. Since it depends on the orientation of $\R^4$ and $*^2=1$, we have $\pm1$ eigenspaces $\Lambda_\pm^2$ of $*$ of equal dimension. The representation of $\SO(4)$ on the space of 2-forms is reducible to each eigenspace and clearly contains $\Z_2=\{1,-1\}$ in the kernel. Counting dimensions, we may introduce two compatible complex structures on $\R^4$ to further deduce the identity $\SO(4)=\SU(2)\times\SU(2)\,/\Z_2$. We thus find $\sol(4)=\sol(3)\oplus\sol(3)$.

Now let $M$ be a connected oriented Riemannian 4-manifold and let us continue with the same notation as above. Then we have a star or Hodge operator $*_{_M}$ on $M$ which, moreover, commutes with covariant differentiation. Hence we have parallel subbundles:
\begin{equation}\label{decompselfandantiself}
 \Lambda^2T^*M=\Lambda^2_+\oplus\Lambda^2_-\ .
\end{equation}

A similar picture as the one from section \ref{sec1:subone} then follows for the principal $\SO(3)$-bundles $P_\pm\lrr M$ of oriented and \textit{$\sqrt{2}$-orthonormal}, i.e. orthogonal and norm $\sqrt{2}$, coframes of $\Lambda^2_\pm$. The group acting is $\SO(3)$ since the rank of $\Lambda^2_\pm$ is 3. By the last term \textit{oriented} we just mean some choice made of one of the two connected-components of the bundle of $\sqrt{2}$-orthonormal coframes of each of those vector bundles associated to $M$.

The spaces $P_\pm$ are nevertheless transformed by $\SO(4)$ under the right-action. Choosing any oriented orthonormal coframe $e=(e^4,\ldots,e^7)\in F_\circ^*M$, we then have two new coframes for the vector bundles of \textit{self-dual} and \textit{anti-self-dual} 2-forms, respectively\footnote{These coframes will be useful later but in separate moments, hence we only introduce the $+$ or $-$ on the $e^i,\ i=1,2,3$, or in other objects, when necessary. We adopt the common notation $e^{\alpha\beta}=e^\alpha\wedge e^\beta$.}:
\begin{equation}\label{self-dualformsandantiself-dualforms}
 e^1=e^1_\pm=e^{45}\pm e^{67}\ ,\qquad e^2=e^2_\pm=e^{46}\mp e^{57}\ ,\qquad e^3=e^3_\pm=e^{47}\pm e^{56}\ .
\end{equation}
This induced coframe $(e^1,e^2,e^3)$ in fact determines invariantly the above choice of $P_\pm$, and hence confirms that the $\Lambda^2_\pm T^*M\lrr M$ are oriented vector bundles. Let us prove this on just one space, say $P_+$, for clarity. Any oriented coframe on $M$ equals $e\cdot g$, for some $g\in\SO(4)$, and any $\sqrt{2}$-orthonormal oriented coframe of $\Lambda^2_+$ is of the previous type, by linear algebra. The orientation of $((e\cdot g)^1,(e\cdot g)^2,(e\cdot g)^3)=(e^1\cdot g,e^2\cdot g,e^3\cdot g)=(e^1,e^2,e^3)\cdot \tilde{g}$ is fixed by $g\in\SO(4)$ since this group is connected and acts transitively. Then
\begin{equation}\label{theprincipalbundles}
 p_+:F_\circ^*M\lrr P_+\qquad\quad\mbox{and}\quad\qquad p_-:F_\circ^*M\lrr P_-
\end{equation}
are equivariant maps defined by $p_\pm(e)=p_\pm(e^4,e^5,e^6,e^7):=(e^1,e^2,e^3)$. The kernel of the group homomorphism $g\mapsto\tilde{g}$ is a normal subgroup $H$, containing $\{1,-1\}$, such that $\SO(4)/H=SO(3)$. We may say $H\simeq\SU(2)$. Hence the orientation is well-defined by the choice in (\ref{self-dualformsandantiself-dualforms}$_+$).

The induced connections on $P_\pm$ are again denoted by an $\omega=\omega_\pm\in \Omega^1_{P_\pm}(\sol(3))$, although now given by $\na p=p\,p^*\omega$ where $p=p_\pm\circ s$ and $s:U\subset M\lrr F_\circ^*M$ is any local section as before and
\begin{equation}\label{connectionformofselforantiself-dualbundle}
 \omega=\left[\begin{array}{ccc}
              0 & -\omega^3 &\omega^2 \\
              \omega^3& 0 &-\omega^1 \\
              -\omega^2 &\omega^1 &0
              \end{array} \right] \qquad\mbox{with}\qquad\begin{cases}
  p^*\omega^1=\omega^6_7\pm\omega^4_5\\
  p^*\omega^2=\omega^7_5\mp\omega^6_4 \\ 
  p^*\omega^3=\omega^5_6\pm\omega^4_7
\end{cases}\ .
\end{equation}
The curvature tensor $R^{\Lambda^2}$ satisfies $R^{\Lambda^2}_\pm\,p=p\,p^*\rho$ for a new 2-form also denoted by $\rho\in \Omega^2_{P_\pm}(\sol(3))$. Next we define the tautological form $\eta=p_\pm(\theta^4,\ldots,\theta^7)$ as the push-forward by $p_\pm$ of the soldering form components. This 2-form is abbreviated henceforth as $\eta=(e^1,e^2,e^3)$, without risk of confusion. We find an important result, which holds on the manifold $P_\pm$ by equivariance and follows consistently with all previous structure equations.
\begin{prop}\label{cartan2eBianchi2erho}
On $P_\pm$ we have
\begin{equation}\label{cartan2}
 \dx\eta=\eta\wedge\omega
\end{equation}
and
\begin{equation}\label{Bianchi2}
 0=\eta\wedge(\omega\wedge\omega+\dx\omega)=\eta\wedge\rho
\end{equation}
where $\rho$ is the curvature 2-form
\begin{equation}\label{curvatureformofselforantiself-dualbundle}
 \rho=\dx\omega+\omega\wedge\omega=
 \left[\begin{array}{ccc}
              0 & -\rho^3 &\rho^2 \\
              \rho^3& 0 &-\rho^1 \\
               -\rho^2 &\rho^1 &0
        \end{array} \right]      \quad\mbox{with}\quad
      \begin{cases}
      \rho^1=\rho^6_7\pm\rho^4_5\\
      \rho^2=\rho^7_5\mp\rho^6_4 \\ 
      \rho^3=\rho^5_6\pm\rho^4_7
     \end{cases}\ .
\end{equation}
\end{prop}
\begin{proof}
 We find indeed $R^{\Lambda^2}_\pm\,p=p\,p^*\rho$ and the formulae  $\dx\omega^3+\omega^1\wedge\omega^2=\rho^5_6\pm\rho^4_7$, etc.
\end{proof}

Let us recreate the celebrated representation theory of the Riemannian curvature tensor, which is due to Singer and Thorpe, cf. \cite{Besse}.

One can prove that the curvature tensor $R^\na$ of the Riemannian 4-manifold $M$ is symmetric when it is seen as a section of $S^2(\Lambda^2 T^*M)$, by the identity in (\ref{Bianchi2}). 

Let $\{e_4,e_5,e_6,e_7\}$ be a dual frame of the above. One defines a map $\calR:\Lambda^2\lrr\Lambda^2$ by
\begin{equation}\label{curvature4}
  \langle \calR(e_\alpha\wedge e_\beta),e_\gamma\wedge e_\delta\rangle
  =-\langle R^\na(e_\alpha,e_\beta)e_\gamma,e_\delta\rangle
  =R^\na_{\alpha\beta\gamma\delta}\ .
\end{equation}
Then there are invariantly defined maps $A,B,B^*,C$ respecting the
decomposition \eqref{decompselfandantiself}, i.e. such that
\begin{equation}\label{curvaturematrix}
\calR=\left[\begin{array}{cc}
                     A & B \\ B^* & C
                    \end{array}\right]\ .
\end{equation}
\begin{lemma}[Singer-Thorpe] \label{LemmaSingerThorpe}
 The map $\calR$ is symmetric, $B$ corresponds to the traceless part of the Ricci tensor $\ric=\sum_{\alpha=4}^7 \langle R(\ ,e_\alpha)e_\alpha,\ \rangle$ and $\tr{A}=\tr{C}=\frac{1}{4}\mathrm{tr_g}{\ric}=\frac{1}{4}\Scal_M$.
\end{lemma}
\begin{proof}
By all definitions, notice
\[ R^\na_{\alpha\beta\gamma\delta}
   =\langle\calR(e_\alpha\wedge e_\beta),e_\gamma\wedge e_\delta\rangle
  =-\langle R^\na_{e_\alpha,e_\beta}e_\gamma,e_\delta\rangle
  =-\langle R^\na_{e_\alpha,e_\beta}e^\gamma,e^\delta\rangle
  =\rho_\gamma^\delta(e_\alpha,e_\beta) \ .\]
  Using the frame $e^1_+,e^2_+,\ldots,e^3_-$, we may clearly write $\rho^i_+ =\sum_{j=1}^3\tilde{a}^i_je^j_++\tilde{\tilde{b}}^i_je^j_-$ for some scalar functions $\tilde{a}^i_j,\tilde{\tilde{b}}^i_j$.
 On the other hand, we have $\calR e^i_+=\sum_j\,\frac{1}{2}\calR_{ij}e^j_++\frac{1}{2}\calR_{i\bar{j}}e^j_-$ where $\calR_{ij}$ follows linearly from \eqref{curvature4}. With the dual frame $p_\pm(e_4,\ldots,e_7)=(e_{\pm,1},e_{\pm,2},e_{\pm,3})$ we then have $e^i_\pm(e_{\pm,j})=2\delta^i_j,\ e^i_\pm(e_{\mp,j})=0$,\ $\forall i,j=1,2,3$, and computations with \eqref{curvature1}, \eqref{curvatureformofselforantiself-dualbundle} yield
\begin{equation}\label{componentsofcurvaturetensor}
\tilde{a}^i_j=\frac{1}{2}\rho_+^i(e_{+,j})=-\frac{1}{2}\calR_{ij}\qquad\qquad 
\tilde{\tilde{b}}^i_j=\frac{1}{2}\rho_+^i(e_{-,j})
=-\frac{1}{2}\calR_{i\bar{j}}\ .
\end{equation}
In particular $-\tilde{a}$ is the matrix of $A$ and $-\tilde{\tilde{b}}$ is the matrix of $B^*$. Also $\rho^i_-=\sum_{j=1}^3{\tilde{b}}^i_je^j_++
\tilde{c}^i_je^j_-$, for some coefficients, and the same holds for $\calR e^i_-=\sum_j\,\frac{1}{2}\calR_{\bar{i}j}e^j_++\frac{1}{2}\calR_{\bar{i}\bar{j}}e^j_-$. Again one shows $\tilde{c}^i_j=+\frac{1}{2}\calR_{\bar{i}\bar{j}}$ and the three identities $2{\tilde{b}}^i_j=\rho^i_-(e_{+,j})=\calR_{\bar{i}j}=-2\tilde{\tilde{b}}^i_j$ which yield $\calR_{\bar{i}j}=\calR_{i\bar{j}}$.
By \eqref{Bianchi2} it is immediate that $a$ and $c$ are symmetric. For instance, on the self-dual part, we find $0=e^2\wedge\rho^3-e^3\wedge\rho^2=+2(\tilde{a}^3_2-\tilde{a}^2_3)e^{4567}$. This implies the whole symmetry of $\calR$. In particular $B^*$ is the adjoint of $B$. Recurring to the first Bianchi identity \eqref{firstBianchiidentity}, further computations on the above coefficients yield the relations with the tensor $\ric$.
 \end{proof}
Henceforth the curvature of the vector bundle of self-dual 2-forms encodes half of the Riemannian curvature tensor of $M$. A few lines of computation will show that $M$ is Einstein, i.e. the Ricci tensor is a multiple of the metric tensor, if and only if $B=0$. In other words, $M$ is Einstein if and only if $*\calR=\calR*$. If this is the case, then clearly orthogonal planes in $TM$ have the same sectional curvature. And reciprocally.

The invariant theory of $\mathrm{SO}(4)$ lets us define the tensors $W_+=A-\frac{1}{3}\tr{A}$ and $W_-=C-\frac{1}{3}\tr{C}$, which are called the \textit{self-dual} and \textit{anti-self-dual} \textit{Weyl} tensors of $M$. The so-called \textit{Weyl} tensor $W=W_++W_-$ is conformally invariant, since that is certainly the case with the star operator and each $W_\pm$ component does preserve the $\Lambda^2_\pm$.

The Riemannian manifold $M$ is \textit{self-dual} if $W=W_+$ and  \textit{anti-self-dual} if $W=W_-$. Clearly the former condition reads also as ($s=\frac{1}{12}\Scal_M=\frac{1}{3}\tr{A}=\frac{1}{3}\tr{C}$):
\begin{equation}\label{self-dualcurvatureequation}
  \mbox{(SD)}\qquad W_-=0\qquad\Longleftrightarrow\qquad\forall m\in M,\ \exists s\in\R\: :\ \rho^i_-=se^i_-+\sum_{j=1}^3\tilde{b}^i_je_+^j,\ \forall i\ ,
\end{equation}
whereas the latter corresponds with $\rho^i_+=-se^i_++\ldots$. 
In any dimension, if $M$ is Einstein, then $s$ is known to be a constant.

\subsection{{$\gdois$} structures}
\label{sec1:subthree}

$\gdois$ structures are well-known today and amount to 3-forms of special kind on a 7-dimensional manifold. One way to describe them is precisely within the above setting of distinguished 2-forms. Let us continue with the notation for self-duality from \eqref{self-dualformsandantiself-dualforms}, but now on some oriented Euclidean 4-space, say a \textit{horizontal} direction, which we complement with a 3-dimensional Euclidean space given by an orthonormal coframe, i.e. a set of three independent linear forms $f^1,f^2,f^3$, for the \textit{vertical} direction. Of course, we obtain a corresponding metric $\mathrm{g}=\mathrm{g}_V+\mathrm{g}_H$ in 7 dimensions. Then a linear $\gdois$ structure is defined on the direct sum vector space, just as in \cite{BrySal,Sal3}, by
\begin{equation}\label{standarphi}
 \phi= 
 \lambda^3f^{123}\mp\lambda\mu^2(f^1\wedge e^1+f^2\wedge e^2+f^3\wedge e^3)  \ .
\end{equation}
In the above we continue to abbreviate $e^i=e^i_\pm$. The coefficients $\lambda^3,\lambda\mu^2$ appearing are dependent on real scalars $\lambda,\mu$. A study of 3-forms of special type gives that the group of automorphisms of $\phi$, $\gdois$, is a simply-connected, compact, simple, 14 dimensional Lie subgroup of $\SO(7)$, with such special orthogonal group referring to some new metric $\mathrm{g}_\phi$ (cf. \cite{Bryant1}). An orientation form $o=\Vol_{\mathrm{g}}=f^{123}e^{4567}$ may be fixed once and for all, because the $\phi$ induced orientation is invariant by continuity on $\lambda,\mu$ in some open interval. The metric $\mathrm{g}_\phi$ is given, for some $m\in\R$ yet to be determined, and for any vectors $u,v$, by the well-known identity
\begin{equation}
 u\lrcorner\phi\wedge v\lrcorner\phi\wedge\phi
            = \pm6\langle u,v\rangle_{\phi}mo   \ .
\end{equation}
In the case of \eqref{standarphi}, after some lengthy but straightforward computations with the dual frame, we find the following result.
\begin{lemma}
 The frame $f_1,f_2,f_3,e_4,e_5,e_6,e_7$ is orthogonal and satisfies $\langle e_\alpha,e_\alpha\rangle_\phi=\frac{\lambda^3\mu^6}{m}$ and $\langle f_i,f_i\rangle_\phi=\frac{\lambda^5\mu^4}{m}$.
\end{lemma}
Indeed the definitions induce a unique metric irrespective of $\pm$. Now it follows that
\[ m^2=\frac{1}{\|f^{123}e^{4\cdots7}\|^2_{\phi}} =\frac{\lambda^{15}\mu^{12}}{m^3}\frac{\lambda^{12}\mu^{24}}{m^4}  \]
and hence the value of $m=\lambda^3\mu^4$. Also the metric and canonical volume form are:
\begin{equation}
 \mathrm{g}_\phi=\lambda^2\mathrm{g}_V+\mu^2\mathrm{g}_H\ ,\quad\qquad \Vol_{\mathrm{g}_\phi}=mo=\lambda^3\mu^4\Vol_{\mathrm{g}} \ .
\end{equation}
The orientations $o$ and $mo$ agree if and only if $\lambda>0$. We fix $\mu>0$ for convenience. Finally the star operator $*_\phi$ for $\mathrm{g}_\phi$ gives
\begin{equation}\label{standarphistandardpsi}
\begin{cases}
\qquad\quad\:\ \ \ \phi\ =\ \lambda^3f^{123}\mp\lambda\mu^2(f^1\wedge e^1+f^2\wedge e^2+f^3\wedge e^3)\\
\   \psi\ :=\ *_\phi\phi\ =\ \mu^4e^{4567}-\lambda^2\mu^2(e^1\wedge f^{23}+e^2\wedge f^{31}+e^3\wedge f^{12})
\end{cases}\ .
\end{equation}

Since the compatibility between the 3- and 4-dimensional subspace orientations is quite arbitrary, we comment on a further detail. It is quite natural that one starts with his own choice of a frame of self-dual 2-forms. For instance, say we pick $e^1,e^2,-e^3$ (or any other non-orientation preserving transformation in $\Lambda^2_+$). Then we may reverse the signs of $f^3$ and $\lambda$ in order to have the same orientation, $mo$, but the metric induced from the new 3-form \eqref{standarphi} will be of signature $(3,-4)$, a so-called $\tilde{\mathrm{G}}_2$ metric, where the automorphisms Lie group is now the non-compact dual of $\gdois$. In order to have a positive definite metric we would have to start by reversing the sign in the present $-\lambda\mu^2$ coefficient in \eqref{standarphi}. For example, without further ado, we see the $\gdois$ structure $f^{123}+f^1e_+^1+f^2e_+^2-f^3e_+^3$ is used in celebrated references such as \cite{Bryant1,Bryant2,FriIva1,FriIva2,Joy}.

A $\gdois$ \textit{structure} on a 7-dimensional manifold $X$ is given by a smooth 3-form $\phi\in\Omega^3_X$ of the form \eqref{standarphi} in some given coframe $f^1,\ldots,e^7$. Then there is an induced metric $\mathrm{g}_\phi$ and compatible orientation on $X$, as we have seen fibre-wise and for similar reasons the same must hold globally. The structure is furthermore reducing the holonomy of the Levi-Civita connection $\na$ of this metric to $\gdois$ if and only if $\na\phi=0$. That is, any endomorphism of $T_xX$ induced by parallel displacement over a contractible loop around $x$ lies in the Lie group. Such a structure is called \textit{parallel} or \textit{1-flat}. A theorem of Fern\'andez and Gray asserts this is equivalent to $\phi$ being harmonic, cf. \cite[Theorem 5.2]{FerGray}.

The classification of $\gdois$ structures is further developed in \cite{FerGray} and \cite{Bryant2}. It depends on four forms $\tau_i\in\Omega^i_X$ for $i=0,1,2,3$, which appear fibre-wise in $\Lambda^iT^*X$ as $\gdois$-modules $W_i$ of dimensions, respectively, $1,7,14,27$. While the first two representation spaces $W_0,W_1$ are obvious, the third one is $W_2=\g_2=\{\tau_2:\ \tau_2\wedge\phi=\mp*_\phi\tau_2\}$ and the fourth one is $W_3=\{\tau_3:\ \tau_3\wedge\phi=\tau_3\wedge\psi=0\}$. The forms indeed exist and appear in (recall $\psi=*_\phi\phi$)
\begin{equation}\label{classificationequationsofG2structures}
\begin{cases}
 \dx\phi=\tau_0*_\phi\phi+\frac{3}{4}\tau_1\wedge\phi+*_\phi\tau_3\\
 \dx\psi=\tau_1\wedge\psi+\tau_2\wedge\phi
\end{cases}\ .
\end{equation}
Equations $\dx\phi=0$ and $\dx*_\phi\phi=0$, respectively, are those of a \textit{calibrated} and \textit{cocalibrated} $\gdois$ structure. As said above, having both conditions is the same as $\na\phi=0$. Like many authors we also reserve the name $\gdois$-\textit{manifold} for the parallel case. If $\dx\phi=\tau_0\,\psi$ with $\tau_0\neq0$ a constant, then we have a \textit{pure type} $W_0$ or \textit{nearly parallel} structure, cf. \cite{Agri1}. For each $i$, the structures are called of \textit{pure type} $W_i$ if the only non-zero component is $\tau_i$. Pure type $W_1$ is the same as locally conformally parallel, since $\tau_1$ must be closed, i.e. locally exact.

\section{The Bryant-Salamon {$\gdois$}  manifolds}
\label{sec2}

\subsection{Structure equations for {$X_+$} and {$X_-$}}
\label{sec2:subsection1}

This section is based on the famous construction of $\gdois$ structures found in \cite{Sal3,Sal4,BrySal}. We give a new description of their fundamental equations and, moreover, we find the respective torsion forms, in Theorem \ref{torsionformsofX_pm} below.

The manifolds $X_\pm=\Lambda_\pm^2T^*M=P_\pm\times_{\SO(3)}\R^3$, where the representation is the canonical one, are natural vector bundles associated to a given oriented Riemannian 4-manifold $M$. Such manifolds carry many rich $\gdois$ structures. We shall treat the $\pm$ cases simultaneously, occasionally forgetting the subscript notation. This shall be the case of the 3-form $\phi$, over $X_\pm$, which is defined as follows assuming much of the notation from previous sections.

A point $x\in X_\pm$ may be written as $x=pa^t$, where $p=(e^1,e^2,e^3)$ constitutes a coframe of self- or anti-self-dual forms and $a=(a^1,a^2,a^3)$ is a vector of $\R^3$. Then the 2-form $\eta$ from Proposition \ref{cartan2eBianchi2erho} induces another tautological 2-form, $\eta a^t$, well-defined on $X_\pm$. As well as the scalar function $r=\dfrac{1}{2}\|\eta a^t\|_{_M}^2=aa^t$. By \eqref{cartan2}, we have
\begin{equation}\label{derivadaouformaexacta}
 \dx(\eta a^t)=\eta\wedge(\omega a^t+\dx a^t)=\eta\wedge f^t 
 \end{equation}
where
\begin{equation}
 f=\dx a+a\omega^t=\dx a-a\omega \ . 
\end{equation}
Using either this identity or the pullback connection to $X_\pm$ from $\na$ on $M$, we find
\begin{equation}\label{derivadaexterior_r}
 \dx r=2fa^t \ .
\end{equation}
\begin{rema}
 With the intent of easing the reading and no fear of inducing much confusion, from now on we abbreviate notation by dropping the wedge product symbol.
\end{rema}
Next we introduce a diligent tool to deal with several computations. 
Consider the linear map which sends $\alpha\in\Omega^k(\R^3)$, $k\geq0$, to the $\sol(3)$-valued $k$-form $\check{\alpha}$ exactly in the shape of the matrix $\omega=(\omega^1,\omega^2,\omega^3)^\vee$ in \eqref{connectionformofselforantiself-dualbundle}. This is,
\begin{equation}\label{littletool}
 \mbox{if}\ \alpha=(\alpha^1,\alpha^2,\alpha^3),\ \mbox{then}\ \check{\alpha}=\alpha^\vee=
 \left[\begin{array}{ccc}  0 & -\alpha^3 & \alpha^2 \\
  \alpha^3 &0 & -\alpha^1 \\ -\alpha^2 & \alpha^1 &0   \end{array}\right]\ .
\end{equation}
In coherence with our notation we also\footnote{We keep the notation for $\omega$ and $\rho$, the only two exceptions, everywhere referring the matrices defined earlier.} have $\rho=\check{\rho}$. We let $\cdot^\wedge$ denote the left inverse map, defined for any matrix $A$ by $A^\wedge=(a_{32},-a_{31},a_{21})$. We have $(A^\wedge)^\vee=A$ if and only if $A$ lies in the orthogonal Lie algebra. The following identities are trivial to check:
\begin{equation}\label{wedgeveeidentities}
 (\check{\alpha}\check{\delta})^\wedge=(\alpha^1\delta^2,-\alpha^1\delta^3,\alpha^2\delta^3)
 \quad\quad\mbox{and}\quad\quad(\alpha\check{\delta})^\vee=\check{\alpha}\check{\delta}-
 (-1)^{\deg\alpha\deg\delta}\check{\delta}\check{\alpha}\ .
\end{equation}

Returning to our $\gdois$ matter, the components $f=(f^1,f^2,f^3)$ give us the required base of 1-forms with which one defines a structure $\phi$ in the same fashion as \eqref{standarphi}. We define $\beta=f^{123}$ and $\vol=e^{4567}$ since in fact this is the pullback to $X_\pm$ of the volume form of $M$. Henceforth $\phi=\lambda^3f^{123}\mp\lambda\mu^2\eta f^t=\lambda^3\beta\mp\lambda\mu^2\dx(\eta a^t)$ where $\lambda,\mu$ are scalar functions on $X_\pm$, cf. \eqref{standarphistandardpsi}. Also $\psi=\mu^4\vol-\lambda^2\mu^2\eta h^t$ where the 2-form $h$ is defined by $h=(\check{f}\check{f})^\wedge=(f^{23},f^{31},f^{12})$; notice 
$\check{h}=-f^tf=\check{f}^2=\frac{1}{2}(f\check{f})^\vee$.
\begin{prop}
 We have
\begin{equation}\label{structureG2BryantSalamon}
 \begin{cases}
  \dx\phi=\dx\lambda^3\,\beta+\lambda^3h\rho a^t\mp\dx(\lambda\mu^2)\eta f^t\\
  \dx\psi=\dx\mu^4\,\vol-\dx(\lambda^2\mu^2)\eta h^t+\lambda^2\mu^2\eta\check{f}\rho a^t
 \end{cases}\ .
\end{equation}
\end{prop}
\begin{proof}
 It is easy to see that $\dx f=-f\omega-a\rho$.
Applying \eqref{wedgeveeidentities} several times, we find $f\omega\check{f}=-h\omega$ and thence
\begin{equation}
 \dx h=(\dx\check{f}\,\check{f}-\check{f}\dx\check{f})^\wedge=(\dx f)\check{f}=-f\omega\check{f}-a\rho\check{f}= h\omega-a\rho\check{f}\ .
\end{equation}
Since $\beta=\frac{1}{3}hf^t$, we have (in fact $h\omega f^t=0$)
\begin{equation}
 \dx\beta=\frac{1}{3}(\dx h\,f^t+h\dx f^t)=\frac{1}{3}(h\omega f^t-a\rho\check{f}f^t
 -h\omega f^t+h\rho a^t)=h\rho a^t 
\end{equation}
and
\begin{equation}
  \dx(\eta h^t)=-\eta\check{f}\rho a^t\ .
\end{equation}
Since $\eta f^t$ is exact, \eqref{derivadaouformaexacta}, the result follows.
\end{proof}

For the solution of several $\gdois$ equations we follow \cite{BrySal,Sal3} and consider $\lambda,\mu$ as functions of the half square-radius $r$.
\begin{prop}
Let us consider the spaces $X_\pm=\Lambda_\pm^2T^*M$ with the generic Bryant-Salamon $\gdois$ structure $\phi$ and assume $\lambda$ and $\mu$ are only dependent of $r$. We have that $\dx\phi=0$ implies the metric of $M$ is Einstein.
\end{prop}
\begin{proof}
The assumption on a function $\zeta$ on $X_\pm$ of being dependent only of $r$ and \eqref{derivadaexterior_r} imply that $\dx\zeta=2\papa{\zeta}{r}fa^t$. The first line of \eqref{structureG2BryantSalamon} thus becomes $\dx\phi=\lambda^3h\rho a^t\mp2\papa{(\lambda\mu^2)}{r}fa^t\eta f^t$. It is now enough to see the case of self-duality, hence with $\rho=\rho_+$. Recall we have seen the Einstein condition is fulfilled with $\rho_+$ having no anti-self-dual terms, i.e. the vanishing of the $\tilde{\tilde{b}}^i_j$ terms in \eqref{componentsofcurvaturetensor}. If $\phi$ is closed, then indeed we must have $B=0$.
\end{proof}
In the following we find the torsion forms introduced in \eqref{classificationequationsofG2structures}.
\begin{teo}\label{torsionformsofX_pm}
Consider the spaces $X_\pm=\Lambda_\pm^2T^*M$ with the generic Bryant-Salamon $\gdois$ structure $\phi$ and assume $\lambda$ and $\mu$ are only dependent of $r$. Assume also that $M$ is anti-self-dual in the case of $X_+$ or self-dual in the case of $X_-$. We thus have $\rho=\mp s\check{\eta}+\rho_{_B}$, as in equation \eqref{self-dualcurvatureequation}, where $\rho_{{}_B}$ is the Einstein component, which interchanges self- with anti-self-duality depending of which case. Then we have:
\begin{enumerate}
 \item[i)] $\tau_0=0$
 \item[ii)] $\tau_1=\frac{2}{3\lambda^2\mu^4}
            \bigl(\papa{(\lambda^2\mu^4)}{r}-s\lambda^4\mu^2\bigr)\dx r$
 \item[iii)] $\tau_2=
    \mp\bigl(\papa{}{r}(\frac{\mu^2}{\lambda^2})-2s\bigr)
     \bigl(\frac{4\lambda^3}{3\mu^2}ha^t\pm\frac{2\lambda}{3}\eta a^t\bigr)$
 \item[iv)] $\tau_3=\mp\lambda^2f\rho_{{}_B}a^t$\,\ and, in particular, $\tau_3=0$ if and only if $M$ is Einstein.
\end{enumerate}
\end{teo}
\begin{proof}
i) Since the wedge of 4-forms with $\phi$ is equivariant, we find an invariant kernel of such map and then deduce $7\tau_0\Vol_{\mathrm{g}_\phi}=(\dx\phi)\phi$. Suppose by hypothesis that $\dx(\lambda\mu^2)=Sfa^t=\frac{1}{2}S\dx r$. Finally,
\begin{equation*}
 \begin{split}
   (\dx\phi)\phi\ & =\ (\lambda^3h\rho a^t\pm S\eta f^tfa^t)(\lambda^3\beta\mp\lambda\mu^2\eta f^t)\\
  & =\ s\lambda^4\mu^2h\check{\eta}a^t\eta f^t-S\lambda\mu^2\eta f^tfa^t\eta f^t\\
  & =\ 0
 \end{split}
\end{equation*}
because $\rho_{_B}\eta=0$,\ \,$\eta f^th\check{\eta}=\beta\eta\check{\eta}=0$,\ \, $\eta\eta^t=\pm6\vol$,\ \,$\eta^t\eta=\pm2\vol.1_3$\,\ and then\,\ $f\eta^t\eta f^tf=\pm2\vol f f^tf=0$. \\
ii) As above, we define three functions $S,T,U$ simply by $\dx(\lambda\mu^2)=Sfa^t$,  $\dx(\lambda^2\mu^2)=Tfa^t$ and $\dx(\mu^4)=Ufa^t$. Note also the identity $f\check{\eta}+\eta\check{f}=0$, which is easy to check and implies $\eta\check{f}\check{\eta}=-f\check{\eta}^2=\pm4f\vol$. Below we will also need $f^tf=-\check{h}$. We have then, by \eqref{structureG2BryantSalamon},
\begin{equation*}
 \begin{split}
   *_\phi\dx\psi\ & =\ *_\phi\bigl((U\vol-T\eta h^t)fa^t \mp\lambda^2\mu^2\eta\check{f}s\check{\eta}a^t\bigr)\\
   & = \ *_\phi\bigl((U-4s\lambda^2\mu^2)\vol fa^t-T\beta\eta a^t\bigr) \\
   & = \ \frac{\lambda}{\mu^4}(U-4s\lambda^2\mu^2)ha^t\mp\frac{1}{\lambda^3} T\eta a^t \ .
 \end{split}
\end{equation*}
Now, it is known that $\tau_1=\frac{1}{3}*_\phi\bigl((*_\phi\dx\psi)\psi\bigr)$\, (cf. \cite{FerGray,FriIva2}). Hence, since $hh^t=0$ and $\eta h^t=h\eta^t$ is a 4-form,
\begin{equation*}
 \begin{split}
 \tau_1 \ & =\ \frac{1}{3}*_\phi\biggl(\bigl(\frac{\lambda}{\mu^4}(U-4s\lambda^2\mu^2)ha^t\mp\frac{T}{\lambda^3}\eta a^t\bigr)
           \bigl(\mu^4\vol-\lambda^2\mu^2\eta h^t\bigr)\biggr) \\
   & =\ \frac{1}{3}*_\phi\bigl(\lambda(U-4s\lambda^2\mu^2)h\vol a^t\pm\frac{T\mu^2}{\lambda}h\eta^t\eta a^t\bigr) \\
   & = \ \frac{1}{3\lambda}*_\phi(\lambda^2U-4s\lambda^4\mu^2+2T\mu^2)h\vol a^t \\
   & =\ \frac{\lambda}{3\lambda^3\mu^4}(\lambda^2U-4s\lambda^4\mu^2+2\mu^2T)fa^t \ .
 \end{split}
\end{equation*}
Since $(\lambda^2U+2\mu^2T)fa^t=\lambda^2\dx(\mu^4)+2\mu^2\dx(\lambda^2\mu^2)=2\dx(\lambda^2\mu^4)$, the result follows.\\
iii) The easiest way to find $\tau_2$, lying in the $\g_2$ representation module, seems to be by using the formula we have just proved. Recalling \eqref{classificationequationsofG2structures} and the previous formula for $\dx\psi$ and checking $h^tf=\beta.1_3$, we have
\begin{equation*}
 \begin{split}
  \mp*_\phi\tau_2\ &=\ \dx\psi-\tau_1\psi\\
    &=\ (U-4s\lambda^2\mu^2)\vol fa^t-T\beta\eta a^t 
  -\frac{1}{3\lambda^2}(\lambda^2U-4s\lambda^4\mu^2+2\mu^2T)fa^t\vol+\\
   &\hspace{3cm}+\frac{1}{3\mu^2}(\lambda^2U-4s\lambda^4\mu^2+2\mu^2T)fa^t\eta h^t\\
  &=\ \frac{1}{3\lambda^2}\bigl(3\lambda^2U-12s\lambda^4\mu^2
  -\lambda^2U+4s\lambda^4\mu^2-2\mu^2T\bigr)\vol fa^t+\\  
   &\hspace{3cm}-T\beta\eta a^t+\frac{1}{3\mu^2}(\lambda^2U-4s\lambda^4\mu^2+2\mu^2T)\eta h^tfa^t\\
   &=\ \frac{1}{3\lambda^2}\bigl(2\lambda^2U-8s\lambda^4\mu^2-2\mu^2T\bigr)\vol fa^t+\frac{1}{3\mu^2}(\lambda^2U-4s\lambda^4\mu^2-\mu^2T)\eta\beta a^t\\
  &=\ (\lambda^2U-\mu^2T-4s\lambda^4\mu^2)(\frac{1}{3\mu^2}\eta\beta+\frac{2}{3\lambda^2}\vol f)a^t\ .
 \end{split}
\end{equation*}
Hence
\begin{equation*}
 \begin{split}
  \tau_2\ &=\ \mp(\lambda^2U-\mu^2T-4s\lambda^4\mu^2)\bigl(\pm\frac{1}{3\mu^2\lambda^3}\eta
  +\frac{2}{3\lambda\mu^4}h\bigr)a^t\\
  &=\ \mp(\frac{\lambda^2}{\mu^2}U-T-4s\lambda^4)\bigl(\pm\frac{1}{3\lambda^3}\eta
  +\frac{2}{3\lambda\mu^2}h\bigr)a^t\\
  &=\ \mp(2\frac{\lambda^2}{\mu^2}\papa{\mu^4}{r}-2\papa{\lambda^2\mu^2}{r}-4s\lambda^4)\bigl( \frac{2}{3\lambda\mu^2}h\pm\frac{1}{3\lambda^3}\eta\bigr)a^t\\
  &= \ \mp\bigl(\papa{}{r}(\frac{\mu^2}{\lambda^2})-2s\bigr)\bigl(\frac{4\lambda^3}{3\mu^2}ha^t
  \pm\frac{2}{3}\lambda\eta a^t\bigr) \ .
 \end{split}
\end{equation*}
iv) Finally, from the formulae above for $\dx\phi$ and $\tau_1$, we find
\begin{equation*}
 \begin{split}
  \tau_3 \ & =\ *_\phi\bigl(\dx\phi+\frac{3}{4}\phi\tau_1\bigr)\\
   & =\ *_\phi\Bigl(\lambda^3h\rho a^t\pm S\eta f^tfa^t+\frac{1}{4\lambda^2\mu^4}(\mp\lambda\mu^2\eta f^t)(\lambda^2U-4s\lambda^4\mu^2+2\mu^2T)fa^t\Bigr)\\
   &=\ \frac{1}{4\lambda\mu^2}*_\phi\bigl(4\lambda^4\mu^2h\rho\mp4S\lambda\mu^2\eta\check{h} \pm(\lambda^2U-4s\lambda^4\mu^2+2\mu^2T)\eta\check{h}\bigr)a^t\\
   &=\ \frac{1}{4\lambda\mu^2}\bigl(4\lambda^4\mu^2*_\phi(\mp sh\check{\eta}+h\rho_{_B})-4S\mu^2\eta\check{f} +\frac{1}{\lambda}(\lambda^2U-4s\lambda^4\mu^2+2\mu^2T)\eta\check{f}\bigr)a^t\\
   &=\ \frac{1}{4\lambda\mu^2}\bigl(-4s\lambda^3\mu^2(f\check{\eta}+ \eta\check{f})\mp4\lambda^3\mu^2f\rho_{_B} +\frac{1}{\lambda}(\lambda^2U+2\mu^2T-4\lambda\mu^2S)\eta\check{f}\bigr)a^t\\
   &=\ \mp\lambda^2f\rho_{_B}a^t\ .
 \end{split}
\end{equation*}
Indeed, $f\check{\eta}+\eta\check{f}=0$ and 
\begin{equation*}
 (\lambda^2U+2\mu^2T-4\lambda\mu^2S)fa^t=2\dx(\lambda^2\mu^4)-4\lambda\mu^2\dx(\lambda\mu^2)=0  \ .
\end{equation*}
So the formula is much simplified.
\end{proof}
We remark that $ha^t$ is also a global 2-form, just as the 2-form $\eta a^t$.

\subsection{New examples of $\gdois$ manifolds}
\label{sec2:subsection2}

With the above theorem we can construct new examples of $\gdois$ structures of eight different and unusual classes. Regarding pure $W_i$, $i=1,2,3$, and other relevant types, we have further observations.

One writes, in general,
\begin{equation}
\tau_1=\frac{2}{3}\bigl(\dx\log(\lambda^2\mu^4)-s\frac{\lambda^2}{\mu^2}\dx r\bigr)\ .
\end{equation}
In the conditions of Theorem \ref{torsionformsofX_pm}, we can indeed find some examples of non-trivial pure type $W_1$ structures, i.e. locally conformally parallel. However, if $12s=\Scal_M<0$, then the structure is only locally conformally parallel, not globally, and in general the induced metric $\mathrm{g}_\phi$ is not complete nor defined on the whole space. Note that $s$ is constant since $\tau_3=0$. Indeed $\tau_2=0$ has a solution: $\lambda=$constant and $\mu^2=\lambda^2(2sr+c_1)$, where $c_1$ is another constant.

Regarding pure type $W_2$ structures, the equation $\tau_1=0$ does not yield so easily. Taking $\lambda$ constant, leads to a complete solution if and only if $\Scal_M\geq0$, giving an answer to the problem. Taking $\mu$ a constant, leads to another solution but hardly with the metric $\mathrm{g}_\phi$ complete. 

We notice that $\tau_1$ and $\tau_2$ are closely related, by the following simple lemma which is just calculus in the variable $r$.
\begin{lemma}\label{lemadas3funcoes}
With $\lambda,\mu>0$, any two of the following conditions imply the third:
 \begin{equation}\label{eqlemadas3funcoes}
  \lambda\mu=c_0\ \mbox{a constant}\ ,\qquad\qquad\tau_1=0\ ,\qquad\qquad\tau_2=0 \ .
 \end{equation}
\end{lemma}
In order to achieve pure type $W_3$ or even $\gdois$-holonomy, one thus assumes \eqref{eqlemadas3funcoes}; equivalently, one assumes the system of equations $\lambda\mu=c_0$ and $\partial_r\mu^2-s\lambda^2=0$. The unique solution is (with $c_1$ another constant):
\begin{equation}\label{coeficientessolucao}
 (\mu(r))^2=(2c_0^2sr+c_1)^\frac{1}{2}\ ,\quad\qquad
 (\lambda(r))^2=c_0^2(2c_0^2sr+c_1)^{-\frac{1}{2}}\ .
\end{equation}
The only existing compact self-dual Einstein 4-manifolds with $s>0$, result due to Hitchin, were pointed out in the original construction of what we have denoted by $X_-$. The following is well-known.
\begin{teo}[Bryant-Salamon, \cite{BrySal,Sal3}]\label{BryantSalamontheorem}
For $M=S^4$ or $M=\C\Proj^2$ with standard metrics, the spaces $\Lambda_-^2T^*M$ have a complete metric with holonomy $\gdois$. 
\end{teo}
We recall that self-dual (SD) scalar-flat 4-manifolds also give rise to interesting complete $\gdois$ structures on $X_-$ by the same method. Raising questions similar to the above for the $\gdois$ structure on anti-self-dual (ASD) metrics, thus pretending that orientation would precede other requirements, we proceed with the study on $X_+$.

Let us resume with the $\Scal_M=0$ condition. The spin compact scalar-flat K\"ahler surfaces were classified in \cite[Proposition 3]{LeB1} and consist of the Calabi-Yau surfaces, the flat torus modulo a finite group, here denoted $M_0$, and the $\C\Proj^1$-bundles over a Riemann surface of genus $>1$ with the local product metric, here $M_1$.
\begin{teo}
\begin{enumerate}
\item[i)] Let $M$ be any complete scalar-flat K\"ahler surface, with the compatible orientation. Then the associated $\gdois$ structure $\phi$ on the manifold $X_+$ is cocalibrated, i.e. $\dx\psi=0$, if and only if $\lambda,\mu$ are constant. In this case, $\phi$ is of pure type $W_3$ and $\mathrm{g}_\phi$ is complete.
\item[ii)] The three classes of manifolds $\Lambda_+^2T^*K3$, where $K3$ denotes any of the homonymous surfaces, $\Lambda_\pm^2T^*M_0$, all admit complete parallel $\gdois$ structures.
\item[iii)] $\Lambda_+^2T^*M_1$ is of pure type $W_3$ and not parallel.
\item[iv)] Both classes of manifolds $M_{2,k}=k\overline{\C\Proj}^2$, with $k\geq6$ (a $k$-many connected sum of conjugate-oriented $\C\Proj^2$s) and manifolds $M_{3,k}=\C\Proj^2\#k\overline{\C\Proj}^2$, with $k\geq14$, all with the scalar-flat ASD metrics des\-crib\-ed in \cite[Theorem A]{LeB3}, admit complete $\gdois$ structures on $\Lambda_+^2T^*M_{i,k}$ $(i=2,3)$ which are of pure type $W_3$ and not parallel.
\end{enumerate}
\end{teo}
\begin{proof}
i) It is well-known that a K\"ahler surface is scalar-flat if and only if it is anti-self-dual (\cite{Derd1}), a local result. We may thus apply Theorem \ref{torsionformsofX_pm} above to get the first part. Since we have $s=0$, it is indeed $\lambda$ and $\mu$ constant by \eqref{coeficientessolucao}, and reciprocally. Completeness follows by completeness of the totally geodesic fibres, by completeness of the base manifold and the Hopf-Rinow Theorem on local product metrics (cf. \cite{Alb18} for details and \cite{BrySal} for a similar argument, which also appears below).\\
ii) The only spin compact cases in i) are $M_0$ and the $K3$ surfaces with Calabi-Yau metric (\cite{LeB1}). Since the latter and $M_0$ are actually Einstein, all torsion tensors in Theorem \ref{torsionformsofX_pm} vanish.\\
iii)  Fibre and base of $M_1$ have opposite sectional curvature, but $M_1$ is not Einstein, so $\tau_3\neq0$.\\
iv) In \cite{LeB3} it is shown that the metrics considered are not Einstein, so $\tau_3\neq0$; again taking $\lambda,\mu$ constant solves equations $\tau_i=0$ for $i=1,2$.
\end{proof}
The classification of compact simply-connected 4-manifolds with scalar-flat ASD metric consists of the $K3$ surfaces and the two classes $M_{2,k}$ and $M_{3,k}$ --- the statement of LeBrun. Bear in mind that we have been considering classes of metrics up to orientation-preserving isometric diffeomorphism.

Determining the holonomy subgroups of $\gdois$ for the manifolds $\Lambda^2_+T^*K3$, which confirms to be $\SU(2)$, is a simple task also accomplished in \cite{Alb18}. Of course, this finding of a $\gdois$ is stated for the sake of completion. The same is true for the flat class $M_0$ in ii) of trivial holonomy.

The next result, partly stated in \cite{BrySal}, is a mirror of the Bryant-Salamon Theorem \ref{BryantSalamontheorem}, but its proof is not. First recall the complex hyperbolic space $\calH^2_\C=\SU(2,1)/\mathrm{U}(2)$, which is a ball in $\C^2$. From \cite{BCGP} we know that it is Einstein and self-dual for the canonical orientation. Let $r_0\in\R^+$ and
\begin{equation}
 D_{r_0,\pm}M=\{ x\in X_\pm:\ \tfrac{1}{2}\|x\|_{_M}^2<r_0\}\subset\Lambda^2_\pm T^*M\ .
\end{equation}
\begin{teo}
For any given $r_0>0$, the real hyperbolic space $\calH^4=SO(4,1)/SO(4)$ and the complex hyperbolic space $\calH^2_\C$, both endowed with standard metrics, are such that the disk bundle manifolds $D_{r_0,\pm}\calH^4$ and $D_{r_0,-}\calH^2_\C$ admit a non-complete metric with holonomy equal to $\gdois$.
\end{teo}
\begin{proof}
First, one considers of course \eqref{coeficientessolucao} and hence may assume $c_0=1$. Hence $\lambda(r)=(2sr+c_1)^{-\frac{1}{4}}$ and $\mu(r)=(2sr+c_1)^{\frac{1}{4}}$ with constant $s<0$; we recall the 3-form is $\phi=\lambda^3\beta-\lambda\mu^2\eta f^t$ and the metric is $\mathrm{g}_\phi=\lambda^2\mathrm{g}_V+\mu^2\mathrm{g}_H$ for both of the base spaces.  Since we must have $2sr+c_1>0$, we see that $c_1=-2sr_0$ and we are left to play with the disk bundles. From \cite{BCGP} we know that $\calH^2_\C$ is Einstein and self-dual for the canonical orientation. The non-completeness of the metric is seen by the length of a radius in the disk fibres. Indeed, taking $x_0\in X_\pm$ with $\sqrt{2}$ norm for the metric on $M$ and the curve $\gamma(t)=tx_0,\ t\in[0,\sqrt{r_0}[$, we have $r_{\gamma_t}=t^2$ and
\[ \int_0^{\sqrt{r_0}}\|x_0\|_\phi\,\dx t =\int_0^{\sqrt{r_0}}\lambda\,\dx t=\frac{1}{(-2s)^{\frac{1}{4}}}\int_0^{\sqrt{r_0}}\frac{\dx t}{(r_0-t^2)^{\frac{1}{4}}}\sim \int_0^{\sqrt{r_0}}\frac{\dx t}{(\sqrt{r_0}-t)^{\frac{1}{4}}}<+\infty\ .  \]
As the fibres are totally geodesic and spherically symmetric, a fibre geodesic exists but it cannot be extended indefinitely. Finally, the holonomy equal to $\gdois$ follows by a main result which is Theorem 3.1 in \cite{Alb18}.
\end{proof}
We remark the vertical \textit{radial} geodesics, for $s=-1$, written $\gamma(t)=\epsilon(t)x_0\in D_{r_0,\pm}M,\ t\in\R$, with $\|x_0\|_{_M}^2=2$, must have the following equation, cf. \cite{Alb18}:
\begin{equation}
 2\ddot{\epsilon}(r_0-\epsilon^2)+3\dot{\epsilon}^2\epsilon=0\ .
\end{equation}
We note the \textit{incompleteness} of the metric is in sharp contrast with the elliptic geometry case. The end of the above proof is accomplished with a general technique, found in \cite{Alb18}, developed with the purpose of computing holonomy on vector bundles with spherically symmetric metrics. This procedure also gives a new proof of the $s>0$ case, i.e. the case of the Bryant-Salamon manifold.

\begin{rema}
It is interesting to see why, after all, the mirror proof of the result about the holonomy group for the two base manifolds with constant $s>0$ does not work for the other two cases with constant $s<0$. To guarantee the holonomy subgroup of $\gdois$ is the whole group, \cite{BrySal} applies a general criterion which says it is sufficient that there do not exist non-trivial parallel 1-forms on the given $\gdois$ parallel manifold.
Following the article, we must first prove our manifolds $\Lambda^2_+$ are not diffeomorphic to $\R^7$. That is true for the real hyperbolic base, a pseudo-sphere, since $\pi_3(\calH^4)\neq0$. But false for the complex hyperbolic ball $\calH^2_\C$ (contrary to the $\C\Proj^2$ case). Also the proof continues with representation theory of the $G$-module $\calP$ of $\na^{\mathrm{g}_\phi}$-parallel 1-form fields, where $G$ is the isometry group of the base manifold. $\calP$ is a vector space which is, in the real case, and should be, in the complex case, of $\dim<7$.  The isometries preserve $\mathrm{g}_\phi$ by construction, hence $G$ acts on $\calP$. For our hyperbolic base spaces, $G=\SO(4,1)$ and $\mathrm{U}(2,1)$, cf. \cite{BCGP}, which of which are the respective mirrors of the elliptic $G=\SO(5)$ and $\mathrm{SU}(3)$. We also note the orthogonal to $\calP$ is not finite dimensional in $\Omega^1_{\Lambda_+^2}$ so we cannot easily argue with it. A few arguments which the reader may check, valid for all cases, tell us that the $G$ action must have irreducible components of $\dim 0,3$ or $4$. In both elliptic cases, that is impossible and further-on implies that $\calP=0$. But in the real hyperbolic case there do exist representations of $\SO(4,1)$ in dimension 4, cf. \cite{BogWhite}.
\end{rema}

\section{$\gdois$ structures on the frame bundle $P_\pm$}
\label{sec3}

Given the oriented Riemannian 4-manifold $M$ from previous sections, we consider another fibre bundle, this time compact, with 3-dimen\-si\-o\-nal fibres and canonical 2-forms. The principal $\SO(3)$-bundle $P_\pm=P_\pm M$ of oriented norm $\sqrt{2}$ orthogonal frames of $\Lambda^2_\pm T^*M$, introduced in section \ref{sec1:subtwo}, may be endowed with a family of natural $\gdois$ structures.

We continue to denote by $\eta=(e^1,e^2,e^3)$ the tautological 2-form field and by $\omega,\rho$ the, respectively, connection 1-form and curvature 2-form fields of $\sol(3)$ matrices, all three globally defined on the total space $P_\pm M$. They are related by $\dx\eta=\eta\omega$ and $\rho=\dx\omega+\omega\omega$ and hence also by $\eta\rho=0$, cf. Proposition \ref{cartan2eBianchi2erho}. One might recall these equations arise equivariantly from the frame bundle of the cotangent bundle of $M$ and its sections, which we now disregard. Indeed $\omega$ is a connection 1-form and has the same value for every coframe of $M$ which induces a given self-dual or anti-self-dual 2-forms coframe.

Using the methods introduced in (\ref{littletool}), we now define
\begin{equation}
\begin{split}
    f=(\omega^1,\omega^2,\omega^3)\qquad\qquad \hat{\rho}=(\rho^1,\rho^2,\rho^3) \\
    \beta=\omega^{123}  \ .   \hspace{33mm}
\end{split}
\end{equation}
The following identities are easy to deduce:
\begin{equation}
\begin{split}
      \frac{1}{2}f\omega=(\omega^{23},\omega^{31},\omega^{12})=(\omega\omega)^\wedge\qquad\qquad \hat{\rho}=\dx f+\frac{1}{2}f\omega\qquad \\
      \omega\hat{\rho}^t=-\rho f^t\qquad\qquad\beta=\frac{1}{6}f\omega f^t\qquad
      \qquad\omega f^tf=2\beta1_3=f^tf\omega  \\
      \eta\omega f^t=f\omega\eta^t \qquad\qquad\omega\omega f^t=0
      \qquad\qquad\qquad
\end{split}
\end{equation}
and
\begin{equation}
-f\rho f^t=f\omega\hat{\rho}^t=\hat{\rho}\omega f^t=2(\rho^1\omega^{23}+\rho^2\omega^{31}+\rho^3\omega^{12})\ .
\end{equation}
It is convenient to see further, the also purely algebraic relations:
\begin{equation}
\begin{split}
& f\rho f^t\eta f^t=-2(\rho^1\omega^{23}+\rho^2\omega^{31}+\rho^3\omega^{12})(e^1\omega^{1}+e^2\omega^{2}+e^3\omega^{3})=-2\beta\hat{\rho}\eta^t=-2\beta\eta\hat{\rho}^t\\
&\qquad\qquad\qquad\qquad\eta\omega f^t\eta f^t=f\omega\eta^t\eta f^t=\pm2\vol f\omega f^t=\pm12\beta\vol\\
&\quad\qquad\qquad\qquad \eta f^t\eta f^t=0\qquad\qquad \eta\hat{\rho}^t\eta f^t=f\eta^t\eta\hat{\rho}^t=
 \pm2\vol f\hat{\rho}^t \ .
\end{split}
\end{equation}
Finally, the announced $\gdois$ structures are given by
\begin{equation}
 \begin{cases}
  \qquad\quad\phi=\lambda^3\beta\mp\lambda\mu^2\eta f^t \\
  \psi=*_\phi\phi=\mu^4\vol-\frac{\lambda^2\mu^2}{2}\eta\omega f^t
 \end{cases}
\end{equation}
with positive scalar functions $\lambda,\mu\in\Omega^0_{P_\pm}$. Again recalling $\eta\rho=0$, let us differentiate the components and then the forms $\phi$ and $\psi$:
\begin{equation}
 \begin{split}
 \dx\beta\ =\ & \frac{1}{6}(\hat{\rho}\omega f^t-f\rho f^t+f\omega\hat{\rho}^t)\ =\ -\frac{1}{2}f\rho f^t\\
 \dx(\eta f^t)\ =\ &\eta\omega f^t-\frac{1}{2}\eta\omega f^t+\eta\hat{\rho}^t\ =\ \eta(\frac{1}{2}\omega f^t+\hat{\rho}^t)\\
 \dx(\eta\omega f^t)\ =\ &\eta(\omega\omega f^t+\rho f^t-\omega\omega f^t-\omega\hat{\rho}^t+\frac{1}{2}\omega\omega f^t)\ =\ -\eta\omega\hat{\rho}^t
 \ =\ \eta\rho f^t\ =\ 0\\
   \dx\phi\ =\ &\dx\lambda^3\,\beta-\frac{\lambda^3}{2}f\rho f^t\mp\dx(\lambda\mu^2)\,\eta f^t \mp\lambda\mu^2\eta(\frac{1}{2}\omega f^t+\hat{\rho}^t)\\
   \dx\psi\ =\ &\dx\mu^4\,\vol-\frac{1}{2}\dx(\lambda^2\mu^2)\,\eta\omega f^t\ .
 \end{split}
\end{equation}
Now we look for the torsion tensors.
\begin{prop}
Let $s=\frac{\Scal_M}{12}$ be the scalar curvature function. We then have:
\begin{equation}
 \tau_0=\pm\frac{6}{7\lambda\mu^2}(\mu^2+2s\lambda^2)\ .
\end{equation}
\end{prop}
\begin{proof}
Recalling the equations for $\rho$ in (\ref{curvaturematrix}), we note the remarkable equation $\eta\hat{\rho}^t=-6s\vol$. With the dimensions of the vertical and horizontal 1-form subspaces in mind, we find
 \begin{equation*}
 \begin{split}
7\tau_0\Vol_\phi =\ & \phi\dx\phi \\
   =\ &\mp\lambda^4\mu^2\beta\eta\hat{\rho}^t\pm\frac{\lambda^4\mu^2}{2}f\rho f^t\eta f^t
       +\frac{1}{2}\lambda^2\mu^4\eta\omega f^t\eta f^t\\
       =\ & \mp\lambda^4\mu^2\beta(\eta\hat{\rho}^t+\hat{\rho}\eta^t)\pm6\lambda^2\mu^4\beta\vol\\
       =\ & \pm6\lambda^2\mu^2(2s\lambda^2+\mu^2)\beta\vol
 \end{split}
\end{equation*}
and the result follows.
\end{proof}
Computations have shown that it is wise to fix $\mu$ and $\lambda$ as constants; otherwise they considerably weigh on the equations and do not seem to lead to any remarkable proposition. In this setting we write a theorem, whose final statement is obtained as usual from $\tau_3=*_\phi\dx\phi-\tau_0\phi$.
\begin{teo}
For any oriented Riemannian 4-manifold $M$, the spaces $P_\pm$ admit a family of $\gdois$ structures defined by the above and the canonical 3-form  $\phi=\lambda^3\beta\mp\lambda\mu^2\eta f^t$. Then we have that $\psi=*_\phi\phi=\mu^4\vol-\frac{\lambda^2\mu^2}{2}\eta\omega f^t$. For any positive constants $\lambda,\mu$, such $\gdois$ structures are always cocalibrated ($\tau_1=\tau_2=0$) and non-calibrated. Moreover
 \begin{equation}
  \tau_3=\lambda^2(*_{{}_M}\hat{\rho})f^t+\frac{1}{7}\Bigl((12s\lambda^2-\mu^2)\eta f^t \pm
  (30s\frac{\lambda^4}{\mu^2}-6\lambda^2)\beta\Bigr) \ .
 \end{equation}
\end{teo}
We remark it is quite demanding to check that $\phi\tau_3=0$ and $\psi\tau_3=0$, as the theory predicts. For the first, one is confronted with the appearance of a 6-form $f\eta^t(*_{_M}\hat{\rho})f^t$, which vanishes. Indeed, between the two $f$ we find a symmetric matrix $\eta^t\hat{\rho}$, essentially the map $A$ or $C$ from (\ref{curvaturematrix}), which one may hence diagonalise. Also checking that $\psi\tau_3=0$ asks for the deduction of an auxiliary result, in which Lemma \ref{LemmaSingerThorpe} and its proof are recalled:
\begin{equation}
\eta\omega f^t(*_{_M}\hat{\rho})f^t=\pm\hat{\rho}f^t\eta\omega f^t=\pm\hat{\rho} f^tf\omega\eta^t=\pm\hat{\rho}\eta^t2\beta=4(\mp\tr{\left\{{}^A_C\right\}})\vol\beta=\mp12s\beta\vol.
\end{equation}
Then one may proceed to verify $\psi\tau_3=0$, with deserved satisfaction.

Recall that $M$ is anti-self-dual (respectively, self-dual) and Einstein if in referring to $P_+$ (respectively, $P_-$) we have $\hat{\rho}=-s\eta$ (respectively, $\hat{\rho}=s\eta$).
\begin{coro}
The $\gdois$ structure $\phi$ is of pure type $W_3$ if and only if $M$ has constant scalar curvature and $\mu,\lambda$ satisfy $\Scal_M=-\frac{6\mu^2}{\lambda^2}$. In this case, $\tau_3\neq0$ since
\begin{equation}
 \tau_3=\lambda^2(*_{{}_M}\hat{\rho})f^t-\mu^2\eta f^t\mp3\lambda^2\beta \ .
\end{equation}
If moreover $M$ is also ASD (SD) and Einstein, then $\tau_3=\pm\frac{1}{2\lambda}(\phi-7\lambda^3\beta)$.
\end{coro}
Now the vanishing of $\tau_3$ implies those curvature restrictions on duality and the Ricci tensor. The reader may deduce the following corollary.
\begin{coro}[cf. \cite{FriKaMoSe}]
The structures $(P_-,\phi)$ for $M=S^4$ or $\C\Proj^2$, such that $s=\frac{\mu^2}{5\lambda^2}$, are nearly parallel. Moreover, $\dx\phi=-\frac{6}{5\lambda}\psi$.
\end{coro}
It is not clear\footnote{Comparing with twistor space, it is not even clear $\SO(5)$ acts transitively on $P_-S^4$.} to the author which nearly parallel structures from the classification in \cite[Tables 1,2,3]{FriKaMoSe} are newly represented by $P_-S^4$ and $P_-\C\Proj^2$.

Clearly the two spaces admit $\gdois$ structures such that $\|\dx\phi\|_\phi$ may be made arbitrarily small or arbitrarily large, but this is a general feature of nearly parallel structures.
Also, again the last result shows a {symmetry breaking} between $P_+$ and $P_-$, i.e. between positive and negative scalar curvature.

The principal $\SO(3)$-bundle connection 1-form $\omega$ is globally defined, so we could well define a $\gdois$ structure with $f$ given by any other permutation of $\omega^1,\omega^2,\omega^3$. How ever this may be done it does not lead to any remarkable results, since then the basic equations have proved to become quite twisted.

We have proved above that cocalibrated $\gdois$ structures are quite abundant, in coherence with \cite[Theorem 1.8]{CrowleyNordstrom}. Regarding 4-dimensional geometry, they appear naturally as, for instance, the celebrated \textit{symplectic} cotangent bundle of every given manifold. Hence there is true motivation for exploring $\gdois$ with a new natural Hamiltonian theory for 4-manifolds.

\vspace{7mm}

\bibliographystyle{plain}
\bibliography{AlbuquerquesBibliography.bib}

\begin{thebibliography}{10}

\bibitem{Agri1}
I.~Agricola.
\newblock The {S}rn\'\i\ lectures on non-integrable geometries with torsion.
\newblock {\em Archi. Mathematicum (Brno) Suppl.}, 42:5--84, 2006.

\bibitem{Alb18}
R.~Albuquerque.
\newblock On vector bundle manifolds with spherically symmetric metrics.
\newblock {\em Ann. Global Anal. Geom.}, 51:129--154, 2017.

\bibitem{Besse}
A.~L. Besse.
\newblock {\em Einstein Manifolds}.
\newblock Springer-Verlag, 1987.

\bibitem{BogWhite}
R.~Bogdanovi\'c and M.~A. Whitehead.
\newblock The representation of the $\mathrm{SO}(4,1)$ group in
  four-dimensional euclidean and spinor space.
\newblock {\em J. Math. Phys.}, 16(400), 1975.

\bibitem{BCGP}
Ch. Boyer, D.~Calderbank, K.~Galicki, and P.~Piccinni.
\newblock Toric self-dual einstein metrics as quotients.
\newblock {\em Commun. Math. Phys.}, 253:337--370, 2005.

\bibitem{Bryant1}
R.~L. Bryant.
\newblock Metrics with exceptional holonomy.
\newblock {\em Annals of Math.}, 126(3):525--576, 1987.

\bibitem{Bryant2}
R.~L. Bryant.
\newblock Some remarks on $\mathrm{G}_2$ structures.
\newblock In {\em Proceedings of the 2004 Gokova Conference on Geometry and
  Topology}, May 2003.

\bibitem{BrySal}
R.~L. Bryant and S.~Salamon.
\newblock On the construction of some complete metrics with exceptional
  holonomy.
\newblock {\em Duke Math. Journ.}, 58(3):829--850, 1989.

\bibitem{CrowleyNordstrom}
D.~Crowley and J.~Nordstr{\"{o}}m.
\newblock A new invariant of {$\mathrm{G}_2$}-structures.
\newblock {\em Geom. Topol.}, 19:2949--2992, 2015.

\bibitem{Derd1}
A.~Derdzi{\'n}ski.
\newblock Self-dual {K\"a}hler manifolds and einstein manifolds of dimension
  four.
\newblock {\em Compositio Mathematica}, 49(3):405--433, 1983.

\bibitem{FerGray}
M.~Fern\'andez and A.~Gray.
\newblock Riemannian manifolds with structure group $\mathrm{G}_2$.
\newblock {\em Ann. Mat. Pura Appl.}, 132(4):19--45, 1982.

\bibitem{FriIva1}
Th. Friedrich and S.~Ivanov.
\newblock Parallel spinors and connections with skew-symmetric torsion in
  string theory.
\newblock {\em Asian Jour. of Math.}, 6(2):303--335, 2002.

\bibitem{FriIva2}
Th. Friedrich and S.~Ivanov.
\newblock Killing spinor equations in dimension 7 and geometry of integrable
  $g_2$-manifolds.
\newblock {\em J. Geom. Phys.}, 48:1--11, 2003.

\bibitem{FriKaMoSe}
Th. Friedrich, I.~Kath, A.~Moroianu, and U.~Semmelmann.
\newblock On nearly parallel $\mathrm{G}_2$-structures.
\newblock {\em J. Geom. Phys.}, 23:259--286, 1997.

\bibitem{Hel}
S.~Helgason.
\newblock {\em Differential Geometry, Lie Groups, and Symmetric Spaces}.
\newblock Academic Press, 1978.

\bibitem{Joy}
D.~Joyce.
\newblock {\em Riemannian {H}olonomy {G}roups and {C}alibrated {G}eometry}.
\newblock Oxford Graduate Texts in Mathematics. Oxford University Press, 2009.

\bibitem{KobNomi}
S.~Kobayashi and K.~Nomizu.
\newblock {\em Foundations of Differential Geometry}, volume 1 and 2.
\newblock Wiley Classics Library, 1996.

\bibitem{LeB1}
C.~LeBrun.
\newblock On the topology of self-dual 4-manifolds.
\newblock {\em Proceedings of the American Mathematical Society},
  98(4):637--640, December 1986.

\bibitem{LeB3}
C.~LeBrun.
\newblock Curvature functionals, optimal metrics, and the differential topology
  of 4-manifolds.
\newblock In S.K. Donaldson, Y.~Eliashberg, and M.~Gromov, editors, {\em
  Different Faces of Geometry}. Kluwer Academic/Plenum, 2004.

\bibitem{Sal3}
S.~Salamon.
\newblock Self-duality and exceptional geometry.
\newblock In {\em “Topology and its Applications, Baku”}, 1987.

\bibitem{Sal4}
S.~Salamon.
\newblock {\em Riemannian geometry and holonomy groups}.
\newblock Pitman research notes in mathematics series. Longman Scientific \&
  Technical, 1989.

\end{thebibliography}

\vspace{6mm} 

\textsc{R. Albuquerque}
\ \ \textbar\ \ 
{\texttt{rpa@uevora.pt}}

\ 

\noindent
Centro de Investiga\c c\~ao em Mate\-m\'a\-ti\-ca e Aplica\c c\~oes

\noindent
Departamento de Matem\'atica da Universidade de \'Evora\\ 
Rua Rom\~ao Ramalho, 59, 7000, \'Evora, Portugal

\vspace{5mm}

\

The research leading to these results has received funding from the People Programme (Marie Curie Actions) of the European Union's Seventh Framework Programme (FP7/2007-2013) under REA grant agreement n\textordmasculine~PIEF-GA-2012-332209.

\end{document}